\documentclass[11pt,reqno]{amsart}

\usepackage[T1]{fontenc}
\usepackage{microtype}
\usepackage[a4paper,margin=1.15in]{geometry}
\usepackage{amsmath,amssymb,mathtools,mathrsfs}

\usepackage{enumitem}
\setlist[enumerate,1]{
	label=(\roman*),
	ref=(\roman*)
}

\usepackage{graphicx}
\usepackage{subcaption}

\allowdisplaybreaks[2]

\theoremstyle{plain}
\newtheorem{theorem}{Theorem}[section]
\newtheorem{lemma}[theorem]{Lemma}
\newtheorem{proposition}[theorem]{Proposition}

\theoremstyle{definition}

\theoremstyle{remark}

\newcommand{\N}{\mathbb{N}}
\newcommand{\R}{\mathbb{R}}
\newcommand{\SigmaM}{\{1,\ldots,m\}^{\N}}
\newcommand{\restr}[2]{\left.#1\right|_{#2}}

\DeclareMathOperator{\diam}{diam}
\DeclareMathOperator{\dist}{dist}

\usepackage[
colorlinks=true,
linkcolor=blue,
citecolor=blue,
urlcolor=blue
]{hyperref}

\begin{document}

\title[Local Scaling and Dimension Distortion]{Local Scaling and Dimension Distortion of Generalized Cantor Functions}

\author{Yuanzhe Shi} \address{Hongshen Honors School, Chongqing University, Chongqing, 401331, P.R. China} \email{yuanzhe.shi@icloud.com\\  shiyuanzhe@stu.cqu.edu.cn}

\author{Zhantu Yang} \address{Hongshen Honors School, Chongqing University, Chongqing, 401331, P.R. China} \email{zhantuyang@foxmail.com}

\author{Jun Jason Luo}  \address{College of Mathematics and Statistics,  Chongqing University, Chongqing, 401331, P.R. China} \email{jun.luo@cqu.edu.cn}

\begin{abstract}
	Let \(\mu\) be a self-similar Cantor measure on \(\mathbb R\) associated with a probability weight vector \(\mathbf p\), let \(K=\operatorname{supp}\mu\), and let \(F\) denote the distribution function of \(\mu\). We characterize the points \(x\in K\) at which the local scaling exponent
	\[
	\lim_{\substack{y\to x, y\in K}}
	\frac{\log |F(y)-F(x)|}{\log |y-x|}
	\]
	exists and assumes a prescribed value. The characterization is formulated in terms of the convergence of the ratio between the accumulated logarithmic mass and geometric scales, together with the sublinear growth of the  endpoint runs. Unlike the classical ternary case, our approach  applies to arbitrary contraction ratios and probability weights. 
	
	As an application, we construct a subset \(M\subset K\) of full Hausdorff measure on which the local scaling exponent of \(F\) is ${h(\mathbf q,\mathbf p)}/{\chi(\mathbf q)}$ and establish the exact dimension-distortion formula
	\[
	\dim_{\mathrm H} F(A)
	=
	\frac{\chi(\mathbf q)}{h(\mathbf q,\mathbf p)}
	\dim_{\mathrm H} A
	\]
	for every \(A\subset M\), where \(\mathbf q\) is the natural probability vector, \(\chi(\mathbf q)\) is the corresponding Lyapunov exponent, and \(h(\mathbf q,\mathbf p)\) is the cross-entropy of \(\mathbf q\) relative to \(\mathbf p\). A three-branch example illustrates the results.
\end{abstract}



\thanks{The research was supported  by the Hongshen Future Scholars Program of Chongqing University.}
		
\maketitle

\section{Introduction}\label{sec:introduction}

The Cantor function, also known as the devil's staircase,  is a canonical example of a continuous singular
function. It is non-decreasing, constant on every complementary interval of
the ternary Cantor set, and has derivative zero at Lebesgue-almost every
point, yet it maps $[0,1]$ onto itself. From a measure-theoretic perspective,
it is the distribution function of the uniform self-similar  measure on the
ternary Cantor set. More generally, if $\mu$ is a non-atomic singular
probability measure on $\mathbb R$ and $F(x)=\mu((-\infty,x])$,
then
$$
F(y)-F(x)=\mu((x,y]) \qquad (x<y).
$$
The local regularity of $F$ is therefore governed by the concentration of
$\mu$ near the point under consideration. This elementary observation links
singular distribution functions with local dimensions of measures and makes
Cantor functions a natural testing ground for interactions among real
analysis, fractal geometry, and symbolic dynamics; see
\cite{Gilman1932,DovgosheyEtAl2006} for classical background.

The usual derivative is often too rigid to describe the behavior of $F$ on
the support of $\mu$, where all of its increase occurs. A more
intrinsic quantity is the logarithmic local scaling exponent
$$
\lim_{\substack{y\to x, y\in K}}
\frac{\log |F(y)-F(x)|}{\log |y-x|},
$$
where $K=\operatorname{supp}\mu$. When this limit exists, it describes the pointwise power-law relation between mass and distance at $x$. It also has a direct geometric consequence. By the general distortion principle developed in \cite{DMRV2006}, a homeomorphism with a constant logarithmic local exponent on a set scales the Hausdorff dimension of every subset by the reciprocal of that exponent. Thus, a pointwise scaling problem leads naturally to an exact dimension-distortion problem.

The differentiability and local regularity of Cantor functions have been
studied extensively. Eidswick \cite{Eidswick1974} characterized the non-differentiability
of the classical Cantor function through the spacing of the digits
\(0\) and \(2\) in ternary expansions, while Dence
\cite{Dence1979} used a related ternary-like coding to study the
one-sided differentiability of generalized Cantor functions.  Darst
\cite{Darst1993,Darst1995} determined the Hausdorff dimension of the
non-differentiability set of the classical Cantor function, establishing
the well-known dimension-squaring formula. For non-symmetric two-branch
constructions, Morris \cite{Morris2002} showed that the Hausdorff dimension is governed by
implicit equations rather than a universal squaring law. Further
developments for ordered self-similar and Moran constructions were
obtained in
\cite{LiXiaoDekking2002,DekkingLi2003,Li2007,YaoZhangLi2009,BaekLi2014}.
In particular, Dekking and Li \cite{DekkingLi2003} recovered the
dimension-squaring law for natural self-similar measures with finitely many branches, while Li \cite{Li2007} considered the case in which each probability weight exceeds the corresponding contraction ratio. Mixed 
weight regimes were subsequently studied by Yao, Zhang and Li
\cite{YaoZhangLi2009} in the binary case and by Baek and Li
\cite{BaekLi2014} for finite-branch constructions, using digit-frequency and
multifractal methods.

One-sided multifractal analysis places these results in a broader
framework. Falconer \cite{Falconer2004} showed that one-sided local
dimensions can detect behavior invisible to their two-sided counterparts
and recovered the dimension-squaring phenomenon for Ahlfors-regular
measures in a general setting. This viewpoint is particularly natural for
distribution functions, since their one-sided increments are precisely
the masses of one-sided intervals. Thermodynamic approaches were developed
by Kesseb\"ohmer and Stratmann \cite{KessebohmerStratmann2009} for
distribution functions of Gibbs measures and by Troscheit
\cite{Troscheit2014} for self-conformal devil's staircases. Recently,
Feng and Miao \cite{FengMiao2025} determined the Hausdorff, packing, and
box dimensions of H\"older non-differentiability sets for self-similar
measures under the open set condition. 

Earlier studies have focused primarily on the non-differentiability sets
of Cantor functions, and many of their results require specific relations
between the probability weights and contraction ratios. In this paper,
we investigate the existence and possible values of pointwise local
scaling exponents for generalized Cantor functions arising as the
distribution functions of self-similar measures with general probability
weights.

To describe our setting, consider an iterated function system (IFS) of contracting similarities   \(\{S_i\}_{i=1}^m\) on $\R$, given by
\[
S_i(x)=r_i x+t_i,\qquad i=1,\ldots,m,
\]
where \(m\ge 2,  \  0<r_i<1\), and \(0=t_1<t_2<\cdots<t_m<r_m+t_m=1\). Throughout, we assume that the intervals \(S_i([0,1])\) are pairwise disjoint. By Hutchinson's theorem \cite{Hutchinson1981}, there exists a unique compact self-similar set \(K\subset[0,1]\) satisfying $K=\bigcup_{i=1}^m S_i(K)$.  In particular,   $\min K=0$ and $\max K=1$. Moreover, for every positive probability vector \(\mathbf p=(p_1,\ldots,p_m)\), there exists a unique self-similar probability measure \(\mu\) satisfying
\begin{equation}\label{eq-selfsimilarmeasure}
\mu=\sum_{i=1}^m p_i\mu\circ S_i^{-1},
\end{equation}
whose support is $K$. We denote by \(F\) the distribution function of \(\mu\), commonly referred to as a generalized Cantor function. That is, $F(x):=\mu([0,x])$.

 For a coding $\omega=(\omega_k)_{k\ge1}\in \SigmaM$ of $x\in K$ and $n\ge 1$, define
$$
A_n(\omega)=-\sum_{k=1}^n\log p_{\omega_k},
\qquad
B_n(\omega)=-\sum_{k=1}^n\log r_{\omega_k}.
$$

We establish a necessary and sufficient symbolic
criterion for the logarithmic local scaling exponent of \(F\) at
\(x\) to exist and equal a prescribed value \(\alpha>0\):
the bulk quotient \(A_n(\omega)/B_n(\omega)\) must converge to
\(\alpha\), while both the left and right endpoint runs must grow sublinearly. The former
describes  the asymptotic balance between mass and geometric scales, whereas the latter
controls the additional endpoint corrections arising from neighboring
cylinders. Neither condition alone is sufficient. This 
separation of the bulk and endpoint effects is a central feature of our
approach. To the best of our knowledge, this is the first criterion that simultaneously captures
both mechanisms for generalized Cantor functions with arbitrary
contraction ratios and probability weights.

Let $s$ be the similarity dimension of
$K$, and let $\mathbf q$ be the natural geometric probability vector. Set
$$
\chi(\mathbf q)=-\sum_{i=1}^m q_i\log r_i,
\qquad
h(\mathbf q,\mathbf p)=-\sum_{i=1}^m q_i\log p_i.
$$

As an application of the criterion, we next  prove that the logarithmic local scaling exponent equals
$h(\mathbf q,\mathbf p)/\chi(\mathbf q)$ at
$\mathcal H^s$-almost every point of $K$. More precisely, there exists
$M\subset K$ with $\mathcal H^s(K\setminus M)=0$ such that
$$
\dim_{\mathrm H}F(A)
=\frac{\chi(\mathbf q)}{h(\mathbf q,\mathbf p)}
\dim_{\mathrm H}A
\qquad\text{for every }A\subset M.
$$
This yields an exact dimension-distortion formula on a full-measure
subset for generalized Cantor functions. In particular, for the standard middle-third Cantor system with
\(\mathbf p=\mathbf q=(1/2,1/2)\), our result recovers, as a special
case, the corresponding result for the classical Cantor function
obtained in \cite{DMRV2006}.

The paper is organized as follows. Section~\ref{sec:preliminaries} introduces
the notation and states the main results. Section~\ref{sec:local-proof} proves
the pointwise criterion, Section~\ref{sec:dimension-proof} establishes the
almost-everywhere limits and the dimension-distortion formula, and
Section~\ref{sec:example} presents a non-uniform three-branch example illustrating the results and the necessity of the sublinear endpoint-run condition.

\section{Main results}\label{sec:preliminaries}

Let  $K$ be the self-similar set induced by the  IFS $\{S_i\}_{i=1}^m$  as in the introduction.   Let $\Sigma=\SigmaM$ be the symbolic space of the IFS, and let $\sigma$ denote the left shift in the sense that $\sigma(\omega_1,\omega_2,\ldots)=(\omega_2,\omega_3,\ldots)$. There is a one-to-one coding map  $\pi:\Sigma\to K$  given by $$\pi(\omega):=\lim_{n\to\infty}S_{\omega_1}\circ\cdots\circ S_{\omega_n}(0),
  \quad \omega=(\omega_1,\omega_2,\ldots)\in\Sigma.$$
  
Let $\mathbf p=(p_1,\ldots,p_m)$ be a probability vector such that
$p_i>0$ for all $i$ and $\sum_{i=1}^{m}p_i=1$. Let $\nu_{\mathbf p}:=\mathbf p^{\mathbb N}$ denote the corresponding
Bernoulli probability measure on $\Sigma$. Then the self-similar measure $\mu$ defined by \eqref{eq-selfsimilarmeasure} is the pushforward
$\mu=\pi_*\nu_{\mathbf p}$ under the coding map $\pi$.
 
Let $s>0$ be the similarity dimension of the IFS, uniquely determined by
\begin{equation}\label{eq:similarity-dimension}
	\sum_{i=1}^{m} r_i^s=1,
\end{equation}
and set $q_i=r_i^s,\ i=1,\ldots,m$. Then $\mathbf q=(q_1,\ldots,q_m)$ is the natural probability vector
associated with the IFS. We adopt the convention
\begin{equation}\label{eq:cross-entropy-decomposition}
	h(\mathbf q,\mathbf p)
	:=
	-\sum_{i=1}^{m} q_i\log p_i
\end{equation}
for the cross-entropy of $\mathbf q$ relative to $\mathbf p$.	

The Shannon entropy of $\mathbf p$ is $h(\mathbf p):=-\sum_{i=1}^{m}p_i\log p_i$,  while the Kullback--Leibler divergence of $\mathbf q$ from $\mathbf p$ is
defined by
$$
D_{\mathrm{KL}}(\mathbf q\Vert\mathbf p)
:=
\sum_{i=1}^{m}q_i\log\frac{q_i}{p_i}.
$$
Since all coordinates of $\mathbf p$ and $\mathbf q$ are positive, these
quantities are finite; the entropy and cross-entropy are positive, whereas
the Kullback--Leibler divergence is nonnegative. It follows that $h(\mathbf q,\mathbf p)=h(\mathbf q)+D_{\mathrm{KL}}(\mathbf q\Vert\mathbf p)\ge h(\mathbf q)$, with equality if and only if $\mathbf p=\mathbf q$.

Let $\chi(\mathbf q)$ denote the Lyapunov exponent of $\mathbf q$, namely,
\begin{equation}\label{eq:entropy-lyapunov}
	\chi(\mathbf q):=-\sum_{i=1}^{m}q_i\log r_i.
\end{equation}
Then $h(\mathbf q)=s\chi(\mathbf q)$, and hence $h(\mathbf q,\mathbf p) = s\chi(\mathbf q) 	+ D_{\mathrm{KL}}(\mathbf q\Vert\mathbf p)$.

For $\omega=(\omega_1, \omega_2,\ldots)\in \Sigma, n\ge 0$,  we define the lengths of the left and right endpoint runs beginning at the $(n+1)$st position by
$$ L_n^-(\omega):=
  \begin{cases}
    \max\{\ell\ge1:\omega_{n+j}=1,  \   j=1,\ldots,\ell\},
      &\text{if this maximum is finite},\\
    0,&\text{otherwise},
  \end{cases}
$$ and
$$L_n^+(\omega) :=
  \begin{cases}
    \max\{\ell\ge1:\omega_{n+j}=m,  \  j=1,\ldots,\ell\},
      &\text{if this maximum is finite},\\
    0,&\text{otherwise}.
  \end{cases}
$$
Here the maximum of the empty set is understood to be $0$; the second case also covers a tail that is identically equal to the relevant endpoint symbol.
Write
$$ R_n(\omega):=  \max\{L_n^-(\omega),L_n^+(\omega)\}.$$

Note that  
$$
A_n(\omega):=-\sum_{k=1}^{n}\log p_{\omega_k},
\qquad
B_n(\omega):=-\sum_{k=1}^{n}\log r_{\omega_k}.
$$

Set $p_{\min}:=\min_{1\le i\le m}p_i, \ r_{\min}:=\min_{1\le i\le m}r_i, \ r_{\max}:=\max_{1\le i\le m}r_i$. Trivially, $p_{\min}\in(0,1)$, and for any $n\ge 1$, 
\begin{equation}\label{eq:Blinear}
	-\log r_{\max}\le \frac{B_n(\omega)}{n}\le -\log r_{\min}.
\end{equation}

Our first main result is the following pointwise scaling criterion.

\begin{theorem}\label{thm:local-scaling}
Let $x=\pi(\omega)\in K$ and let $\alpha>0$. Then
\begin{equation}\label{eq:main-limit1}
  \lim_{\substack{y\to x, y\in K}}
  \frac{\log|F(y)-F(x)|}{\log|y-x|}
  =\alpha
\end{equation}
holds if and only if
\begin{equation}\label{eq:conditions1}
  \lim_{n\to\infty}\frac{A_n(\omega)}{B_n(\omega)}=\alpha,
  \qquad
  \lim_{n\to\infty}\frac{R_n(\omega)}{n}=0.
\end{equation}
\end{theorem}

The two limits in \eqref{eq:conditions1} serve distinct purposes. The
first describes the asymptotic ratio between the logarithmic mass and
geometric scales, whereas the second requires the endpoint-run lengths
to grow sublinearly. This condition rules out linearly long runs along
the extreme branches, which may produce additional endpoint corrections
in the numerator of \eqref{eq:main-limit1}. The example in
Section~\ref{sec:example} shows that the first limit alone is insufficient
and illustrates why the endpoint-run condition is necessary.

Let $K^1$ be the set of endpoints of the complementary intervals of $K$ in $\R$. Note that $\pi(\omega)\in K^1$ if and only if $\omega$ is eventually identically $1$ or eventually identically $m$.  Write $K^*:=K\setminus K^1$. 

The second main result is the   exact formula for Hausdorff dimension distortion under $F$.

\begin{theorem}\label{thm:dimension-distortion}
Let $s$, $\chi(\mathbf q)$, and $h(\mathbf q,\mathbf p)$ be defined by
\eqref{eq:similarity-dimension}--\eqref{eq:entropy-lyapunov}.
Then there exists a set $M\subset K^*$ such that
\begin{equation}\label{eq:full-Hs-set}
  \mathcal H^s(K\setminus M)=0,
\end{equation}
and, for every $A\subset M$,
\begin{equation}\label{eq:generalized-theorem-15}
  \dim_{\mathrm H}F(A)
  =\frac{\chi(\mathbf q)}{h(\mathbf q,\mathbf p)}\dim_{\mathrm H}A.
\end{equation}
\end{theorem}

The distortion factor has a transparent information-theoretic
interpretation. As $D_{\mathrm{KL}}(\mathbf q\Vert\mathbf p)$   increases,  the cross-entropy $h(\mathbf q,\mathbf p)$ increases. Consequently, the dimension-distortion factor ${\chi(\mathbf q)}/{h(\mathbf q,\mathbf p)}$ decreases. 	In particular, ${\chi(\mathbf q)}/{h(\mathbf q,\mathbf p)}
\le 1/s$, 	with equality if and only if $\mathbf p=\mathbf q$.

Theorem~\ref{thm:local-scaling} provides a pointwise criterion for any
prescribed local scaling exponent $\alpha>0$. For $\mathbf q$-typical points, Proposition \ref{prop:typical-exponent} shows that this
exponent is $h(\mathbf q,\mathbf p)/\chi(\mathbf q)$. By contrast, the dimension distortion in Theorem \ref{thm:dimension-distortion} is governed by the reciprocal of that exponent. Thus, the two theorems describe pointwise scaling and global dimension distortion through a pair of reciprocals.

\section{Pointwise scaling criterion}\label{sec:local-proof}

For $\omega=(\omega_1,\omega_2,\ldots)\in\Sigma$ and $n\ge1$, write $\omega|_n:=\omega_1\cdots\omega_n$, and let $\omega|_0$ denote the empty word. For a finite word $u=u_1\cdots u_n$, write
$$
S_u:=S_{u_1}\circ\cdots\circ S_{u_n},
\qquad
r_u:=\prod_{k=1}^{n}r_{u_k},
\qquad
p_u:=\prod_{k=1}^{n}p_{u_k}.
$$
For the empty word, the corresponding products are understood to be equal to $1$. For convenience, we denote by $\underline{i}:=iii\cdots\in \Sigma$   the coding consisting entirely of the symbol $i$, and let
$$
T:=F\circ\pi,
\qquad
c_i:=\sum_{j<i}p_j.
$$

The following two lemmas play a key role in the proof of Theorem \ref{thm:local-scaling}.

\begin{lemma}\label{lem:symbolic-recursion}
	For every $\omega\in\Sigma$, one has
	\begin{equation*}\label{eq:recursion}
		T(\omega)
		=
		c_{\omega_1}
		+
		p_{\omega_1}T(\sigma\omega).
	\end{equation*}
\end{lemma}

\begin{proof}
	Set $x:=\pi(\omega)$. 	By the definition of the coding map, $\pi(\omega)	= 	S_{\omega_1}\bigl(\pi(\sigma\omega)\bigr)\in S_{\omega_1}(K)$. 	
	Since the first-level cylinder sets are pairwise disjoint and ordered from left to
	right, the part of $K$ lying to the left of $x$ consists
	of all the first-level cylinders $S_1(K),\ldots,S_{\omega_1-1}(K)$ 	together with the portion of $S_{\omega_1}(K)$ between its
	left endpoint and $x$. Because $\mu$ is supported on $K\subset [0,1]$, the
	open gaps between these cylinders have zero $\mu$-measure. 	Therefore,
	\begin{align*}
		T(\omega) &= \mu([0,\pi(\omega)])=
		\sum_{i=1}^{\omega_1-1}\mu(S_i(K))
		+
		\mu\left(
		[S_{\omega_1}(0),\pi(\omega)]
		\right)\\
		&= \sum_{i=1}^{\omega_1-1}\mu(S_i(K)) + \mu([S_{\omega_1}(0), S_{\omega_1}(\pi(\sigma\omega))])\\
		&=c_{\omega_1} + p_{\omega_1}\mu([0, \pi(\sigma\omega)]) = c_{\omega_1} + p_{\omega_1}T(\sigma\omega).
	\end{align*}
\end{proof}

\begin{lemma}\label{lem:boundary-tail}
	Let $x=\pi(\omega)\in K$. The following statements hold.
	\begin{enumerate}
		\item 	If $x$ is not a right endpoint of a complementary interval of $K$, then, for every $n\ge0$,
		\begin{equation*}\label{eq:left-tail}
			T(\sigma^n\omega)
			\ge p_1^{L_n^-(\omega)+1}.
		\end{equation*}
		\item If $x$ is not a left endpoint of a complementary interval of $K$, then, for every $n\ge0$,
		\begin{equation*}\label{eq:right-tail}
			1-T(\sigma^n\omega)
			\ge p_m^{L_n^+(\omega)+1}.
		\end{equation*}
	\end{enumerate}
\end{lemma}

\begin{proof}
\emph{(i)}  Since $x$ is not a right endpoint  of a complementary interval of $K$, $L_n^-(\omega)$ is finite. By Lemma \ref{lem:symbolic-recursion},
$$
  T(\sigma^n\omega) = c_1 + p_1T(\sigma^{n+1}\omega) = p_1T(\sigma^{n+1}\omega).
$$
Iterating this identity $L_n^-(\omega)$ times and using $\omega_{n+L_n^-(\omega)+1} > 1$, we obtain
\begin{align*}
  T(\sigma^n\omega) = p_1^{L_n^-(\omega)} T(\sigma^{n+L_n^-(\omega)}\omega)\ge p_1^{L_n^-(\omega)}T(2\underline{1})= p_1^{L_n^-(\omega)}(c_2 + p_2T(\underline{1}))  =p_1^{L_n^-(\omega)+1}.
\end{align*}

\emph{(ii)}  Similarly, $L_n^+(\omega)$ is finite. By Lemma \ref{lem:symbolic-recursion},
\begin{align*}
  1-T(\sigma^n\omega) &= 1-(c_m + p_mT(\sigma^{n+1}\omega)) \\
  &= 1-(1 - p_m + p_mT(\sigma^{n+1}\omega)) \\
  &= p_m(1-T(\sigma^{n+1}\omega)).
\end{align*}
Iterating this identity $L_n^+(\omega)$ times and using $\omega_{n+L_n^+(\omega)+1} <m$, we obtain
\begin{align*}
  1-T(\sigma^n\omega) &= p_m^{L_n^+(\omega)}(1-T(\sigma^{n+L_n^+(\omega)}\omega))\\
  &\ge  p_m^{L_n^+(\omega)}(1 - T\left((m-1)\underline{m}\right)) \\
  &= p_m^{L_n^+(\omega)}(1 - (c_{m-1} + p_{m-1}T(\underline{m})) )\\
  &= p_m^{L_n^+(\omega)}(1 - (1 - p_m)) = p_m^{L_n^+(\omega) + 1}.
\end{align*}
\end{proof}

\begin{proof}[Proof of Theorem~\ref{thm:local-scaling}]
\emph{(i) Non-endpoint case.}   We   consider the case that $x$ is not an endpoint of a complementary interval of $K$, that is $x=\pi(\omega)\in K^*$.

For $y=\pi(\eta)\in K\setminus\{x\}$, let $n=\min\{k\ge1:\omega_k\ne\eta_k\}$ and $u=\omega_1\cdots\omega_{n-1}$. Then  $x=S_u\bigl(\pi(\sigma^{n-1}\omega)\bigr), \  y=S_u\bigl(\pi(\sigma^{n-1}\eta)\bigr)$, and hence, by similarity,
\begin{align*}
  |x-y| &= r_u\left|\pi(\sigma^{n-1}\omega)-\pi(\sigma^{n-1}\eta)\right|, \\
   |F(x)-F(y)| &= p_u\left|T(\sigma^{n-1}\omega)-T(\sigma^{n-1}\eta)\right|.
\end{align*}
Set
\begin{align*}
  w(x,y):=|\pi(\sigma^{n-1}\omega)-\pi(\sigma^{n-1}\eta)|,\quad  z(x,y):=|T(\sigma^{n-1}\omega)-T(\sigma^{n-1}\eta)|.
\end{align*}
Then we can write
$$
Q(x, y) := \frac{\log\left|F(x) - F(y)\right|}{\log\left|x - y\right|} =\frac{A_{n-1}(\omega)-\log z(x,y)}{B_{n-1}(\omega)-\log w(x,y)}.
$$

Since $\omega_n\ne \eta_n$ and the strong separation condition holds,
\begin{equation}\label{eq:w-bounds}
  0<\delta:=\min_{1\le i<j\le m}\dist(S_i(K),S_j(K))
  \le w(x,y)\le \diam K=1.
\end{equation}
Combining this with \eqref{eq:Blinear}, the factor
$$
  \frac{B_{n-1}(\omega)}{B_{n-1}(\omega)-\log w(x,y)} \longrightarrow 1 \quad \text{as} \  y\to x.
$$
It follows that
\begin{equation*}\label{eq:quotient}
\begin{aligned}
  \lim_{\substack{y\to x, y\in K}}
 Q(x, y)  &=\lim_{\substack{y\to x, y\in K}}
    \frac{A_{n-1}(\omega)-\log z(x,y)}
         {B_{n-1}(\omega)-\log w(x,y)}\\
  &=\lim_{\substack{y\to x, y\in K}}
    \left(\frac{A_{n-1}(\omega)}{B_{n-1}(\omega)}+\frac{-\log z(x,y)}{B_{n-1}(\omega)}\right).
\end{aligned}
\end{equation*}

  We now estimate $z(x,y)$. If $\omega_n>\eta_n$, then
\begin{align*}
  z(x,y)
  &=T(\sigma^{n-1}\omega)-T(\sigma^{n-1}\eta) \\
  &\ge T(\sigma^{n-1}\omega)-T(\eta_n\underline{m})\\
  &= c_{\omega_n} + p_{\omega_n}T(\sigma^n\omega)-c_{\eta_n + 1}\\
  &\ge p_{\omega_n}T(\sigma^n\omega)\ge p_{\min}^{L_n^-(\omega)+2} \ge p_{\min}^{R_n(\omega)+2}.
\end{align*}
Similarly, if $\omega_n<\eta_n$, we have
\begin{align*}
  z(x,y)
  &=T(\sigma^{n-1}\eta)-T(\sigma^{n-1}\omega)\\
  &\ge T(\eta_n\underline{1})-T(\sigma^{n-1}\omega)\\
  &= c_{\eta_n} - (c_{\omega_n} + p_{\omega_n}T(\sigma^n\omega))\\
  &\ge p_{\omega_n}\bigl(1-T(\sigma^n\omega)\bigr)  \ge p_{\min}^{L_n^+(\omega)+2}  \ge p_{\min}^{R_n(\omega)+2}.
\end{align*}
Together with $z(x,y)\le1$, this gives  
$$
0\le-\log z(x,y)
\le\bigl(R_n(\omega)+2\bigr)(-\log p_{\min}).
$$

Assume that \eqref{eq:conditions1} holds. Then, using \eqref{eq:Blinear}, we obtain
$$
  0\le \lim_{\substack{y\to x, y\in K}}\frac{-\log z(x,y)}{B_{n-1}(\omega)} \le \lim_{n\to \infty}\frac{(R_n(\omega)+2)(-\log p_{\min})}{(n-1)(-\log r_{\max})} = 0.
$$
This proves \eqref{eq:main-limit1}.

Conversely,  assume that \eqref{eq:main-limit1} holds.  We first prove the necessity of $A_n(\omega)/B_n(\omega)\to \alpha$.  For each $n\ge0$,  write $q = \omega_{n+1}$. We construct a new coding $\eta^{(n)}$ in the following way: if $q < m$, take
$$
    \eta^{(n)} = \omega|_{n}(q+1)\underline{m};
$$
if $q = m$, take
$$
    \eta^{(n)} = \omega|_{n}(q-1)\underline{1}.
$$
Let $y_n=\pi(\eta^{(n)})$. As $n\to \infty$,  $|x-y_n|\le r_{\omega|_{n}} \diam K\le (r_{\max})^n\to 0$, hence $y_n\to x$.

When $q<m$, by Lemma \ref{lem:symbolic-recursion},
$$
  T(\sigma^n\eta^{(n)})
  =c_{q+1}+p_{q+1},
  \quad
  T(\sigma^n\omega)=c_q+p_qT(\sigma^{n+1}\omega)\le c_{q+1},
$$
we have
$$
 T(\sigma^n\eta^{(n)}) - T(\sigma^n\omega) \ge p_{q+1}\ge p_{\min}.
$$

When $q=m$,
$$
  T(\sigma^n\omega)=c_m+p_mT(\sigma^{n+1}\omega)\ge c_m, \quad   T(\sigma^n\eta^{(n)})=T((m-1)\underline{1})=c_{m-1},
$$
we have
$$
 T(\sigma^n\omega) - T(\sigma^n\eta^{(n)}) \ge p_{m-1}\ge p_{\min}.
$$
Thus in either case, we always have
$$
\frac{A_{n}(\omega)}{B_{n}(\omega)-\log \delta}\le Q(x, y_n) \le \frac{A_{n}(\omega)-\log p_{\min}}{B_{n}(\omega)},
$$
where $\delta$ is as in \eqref{eq:w-bounds}. Taking limits on both sides and using \eqref{eq:Blinear},  it concludes that 
$$\lim_{n\to\infty}\frac{A_n(\omega)}{B_n(\omega)}=\lim_{n\to\infty}Q(x, y_n) =  \alpha.$$

We next prove the necessity of the condition
$R_n(\omega)/n\to 0$ by contradiction. Suppose, as before, that
\eqref{eq:main-limit1} holds. If $R_n(\omega)/n\not\to 0$, then there
exist $\varepsilon>0$ and a strictly increasing subsequence
$\{n_j\}_{j\geq 1}$ such that, for every $j$,
$$L_{n_j}^-(\omega)\geq \varepsilon n_j \qquad\text{or}\qquad L_{n_j}^+(\omega)\geq \varepsilon n_j.$$
By passing to a further subsequence, we may assume that the same
alternative occurs for every $j$. Moreover, after moving each $n_j$
backward to the beginning of the corresponding endpoint run, we may
assume that $\omega_{n_j}\neq 1$ in the first case and
$\omega_{n_j}\neq m$ in the second.

First, if the subsequence $\{n_j\}$ satisfies $ R_{n_j}(\omega)=L_{n_j}^-(\omega)\ge\varepsilon n_j, \ \omega_{n_j}\ne1$. Write $l_j=R_{n_j}(\omega)$ and $q=\omega_{n_j}$. Define
$$
    \eta^{(j)} = \omega|_{(n_j-1)}(q-1)\underline{m},\quad 
    y_j:=\pi(\eta^{(j)}).
$$
Then
\begin{equation*} 
\begin{aligned}
  z(x, y_{j}) &= T(\sigma^{n_j - 1}\omega) - T(\sigma^{n_j - 1}\eta^{(j)})= p_qT(\sigma^{n_j}\omega)
  =p_qp_1^{l_j}T(\sigma^{n_j + l_j}\omega) \le p_{1}^{l_j}.
\end{aligned}
\end{equation*}
Combining the already established convergence $A_n(\omega)/B_n(\omega)\to\alpha$ with
\eqref{eq:Blinear}, we obtain
\begin{align*}
  \liminf_{j\to\infty}Q(x,y_j)
  &\ge \alpha+ \liminf_{j\to\infty}
       \frac{-\log z(x, y_j)}{B_{n_j-1}(\omega)}\\
  &\ge \alpha+ \liminf_{j\to\infty}
       \frac{l_j(-\log p_{1})}{(n_j-1)(-\log r_{\min})}\\
  &\ge \alpha+ \varepsilon\frac{\log p_{1}}{\log r_{\min}} > \alpha,
\end{align*}
which contradicts \eqref{eq:main-limit1}.

In contrast, if the subsequence $\{n_j\}$ satisfies $R_{n_j}(\omega)=L_{n_j}^+(\omega)\ge\varepsilon n_j$ and
$\omega_{n_j}\ne m$. Again write $l_j=R_{n_j}(\omega)$ and $q=\omega_{n_j}$. Define
$$
  \eta^{(j)}=\omega|_{(n_j-1)}(q+1)\underline{1},\quad
  y_j:=\pi(\eta^{(j)}).
$$
Similarly,
\begin{align*}
  z(x,y_j)
  &=T(\sigma^{n_j-1}\eta^{(j)})-T(\sigma^{n_j-1}\omega)=c_{q+1}-\left(c_q+p_qT(\sigma^{n_j}\omega)\right)\\
  &=p_q(1 - T(\sigma^{n_j}\omega))=p_qp_m^{l_j}\left(1-T(\sigma^{n_j+l_j}\omega)\right) \le p_m^{l_j}.
\end{align*}
Hence
\begin{align*}
  \liminf_{j\to\infty}Q(x,y_j)  \ge\alpha+\liminf_{j\to\infty}  \frac{-\log z(x,y_j)}{B_{n_j-1}(\omega)}  \ge \alpha+ \varepsilon\frac{\log p_m}{\log r_{\min}} > \alpha,
\end{align*}
again contradicting \eqref{eq:main-limit1}. Therefore, $R_n(\omega)/n\to0$. This completes the proof of both sufficiency and necessity in the non-endpoint case.

\emph{(ii) Endpoint case.}   We consider the case that $x$ is an  endpoint of a complementary interval of $K$, that is $x=\pi(\omega)\in K^1$. If $x$ is a right endpoint of a complementary interval, then there exists a finite word $v$ such that $\omega=v\underline{1}$. Let $N$ be the length of $v$. If $y=\pi(\eta)\in K\setminus\{x\}$ is sufficiently close to
$x$ and $n:=n(x,y)$ is the first position at which $\omega$ and $\eta$ differ. Then $n> N$ and  there exist
$q\in\{2,\ldots,m\}$ and $\gamma\in\Sigma$ such that
$$\omega=v1^{n-N-1}\underline{1},\qquad  \eta=v1^{n-N-1}q\gamma.$$
In this case $\sigma^{n-1}\omega=\underline{1}$ and
$\sigma^{n-1}\eta=q\gamma$, so
$$
  z(x,y)=T(q\gamma)-T(\underline{1})
  =T(q\gamma)\in [p_1,1].
$$
Combining with $\eqref{eq:Blinear}$, we have
$$\lim_{\substack{y\to x, y\in K}}\frac{ - \log z(x, y)}{B_{n-1}(\omega)} = 0. $$
As $y\to x$, we have $n\to\infty$. Since the coding $\omega$ is identically $1$ after the fixed prefix $v$,
$$\lim_{n\to\infty}\frac{A_{n-1}(\omega)}{B_{n-1}(\omega)}  =\frac{\log p_1}{\log r_1}.$$

It follows that
\begin{align*}
  \lim_{\substack{y\to x, y\in K}}Q(x,y)
  &=\lim_{\substack{y\to x, y\in K}}\left(\frac{A_{n-1}(\omega)}{B_{n-1}(\omega)} + \frac{ - \log z(x, y)}{B_{n-1}(\omega)}\right)\\
  &= \lim_{n\to\infty}\frac{A_{n-1}(\omega)}{B_{n-1}(\omega)}
  =\frac{\log p_1}{\log r_1}.
\end{align*}
Since $R_n(\omega)=0$ for $n>N$, both sides of the asserted equivalence
hold with a prescribed value $\alpha$ if and only if
$\alpha=\log p_1/\log r_1$. Hence the equivalence holds at every right
endpoint.

If $x$ is a left endpoint of a complementary interval, then its coding is eventually identically $m$. Replacing
$1,p_1,r_1$ above by $m,p_m,r_m$, respectively in the above argument, we conclude that  both sides
hold if and only if $\alpha=\log p_m/\log r_m$. This completes the proof of the theorem.
\end{proof}

\section{Typical exponents and dimension distortion}
\label{sec:dimension-proof}

Let $\nu_{\mathbf q}:=\mathbf q^{\N}$ denote the Bernoulli probability measure on the symbolic space $\Sigma$ with natural coordinate weights $\mathbf q=(r_1^s,\ldots,r_m^s)$.

\begin{proposition}\label{prop:typical-exponent}
For $\nu_{\mathbf q}$-almost every $\omega\in\Sigma$,
\begin{equation}\label{eq:typical-symbolic-limits}
  \lim_{n\to\infty}\frac{A_n(\omega)}{B_n(\omega)}
  =\frac{h(\mathbf q,\mathbf p)}{\chi(\mathbf q)},
  \qquad
  \lim_{n\to\infty}\frac{R_n(\omega)}{n}=0.
\end{equation}
Consequently, for $\nu_{\mathbf q}$-almost every coding $\omega$, if
$x=\pi(\omega)$, then
\begin{equation}\label{eq:typical-local-exponent}
  \lim_{\substack{y\to x, y\in K}}
  \frac{\log|F(y)-F(x)|}{\log|y-x|}
  =\frac{h(\mathbf q,\mathbf p)}{\chi(\mathbf q)}.
\end{equation}
\end{proposition}

\begin{proof}
For the Bernoulli probability measure $\nu_{\mathbf q}$ on $\Sigma$, by the strong law of large numbers,  there exists a set
$\Omega_0\subset\Sigma$ of full $\nu_{\mathbf q}$-measure such that, for every $\omega\in\Omega_0$, as $n\to \infty$, we have
\begin{align*} 
  \frac{A_n(\omega)}{n}
  \longrightarrow
  \sum_{i=1}^{m}q_i(-\log p_i)
  &=h(\mathbf q,\mathbf p), \\
  \frac{B_n(\omega)}{n}
  \longrightarrow
  \sum_{i=1}^{m}q_i(-\log r_i)
  &=\chi(\mathbf q).
\end{align*}
It follows that
$$
  \frac{A_n(\omega)}{B_n(\omega)}
  \longrightarrow
  \frac{h(\mathbf q,\mathbf p)}{\chi(\mathbf q)}.
$$

We next prove that $R_n(\omega)/n\to 0$ as $n\to\infty$ almost everywhere. Fix $k\in\N$. By the Bernoulli property,
\begin{align*}
  \nu_{\mathbf q}\{\omega:L_n^-(\omega)\ge \lceil n/k\rceil\}
  \le q_1^{\lceil n/k\rceil} \le q_1^{n/k},\\
  \nu_{\mathbf q}\{\omega:L_n^+(\omega)\ge \lceil n/k\rceil\}
  \le q_m^{\lceil n/k\rceil} \le q_m^{n/k},
\end{align*}
where $\lceil n/k\rceil = \min\{t\in\mathbb Z:n/k \le t \}$ denotes the ceiling of $n/k$. Hence
$$
  \nu_{\mathbf q}\left\{\omega:R_n(\omega)\ge \frac{n}{k} + 1\right\}
  \le \nu_{\mathbf q}\{\omega:R_n(\omega)\ge \lceil n/k\rceil\}
  \le q_1^{n/k}+q_m^{n/k}.
$$
Since $0<q_1,q_m<1$, summing over $n$ gives
$$
  \sum_{n=1}^{\infty} \nu_{\mathbf q}\left\{\omega:R_n(\omega)\ge \frac{n}{k} + 1\right\}
  \le \sum_{n=1}^{\infty} \left(q_1^{n/k}+q_m^{n/k}\right)<\infty.
$$
By the Borel--Cantelli lemma, for every fixed $k\in\mathbb N$,
$$
\nu_{\mathbf q}\left\{
\omega:
\frac{R_n(\omega)}{n}
\ge \frac1k+\frac1n
\ \text{i.o.}
\right\}=0.
$$

Set
$$
E_k
:=
\left\{
\omega:\limsup_{n\to\infty}\frac{R_n(\omega)}{n}>\frac2k
\right\}.
$$ 
If $\omega\in E_k$, then there are infinitely many $n$ such that  $\frac{R_n(\omega)}{n}>\frac2k$.
Hence, for infinitely many sufficiently large $n$, we have
$\frac{R_n(\omega)}{n}
>\frac2k
\ge\frac1k+\frac1n$.
It follows that
\begin{equation*} 
	E_k\subset \left\{
	\omega:\frac{R_n(\omega)}{n}\ge\frac1k+\frac1n\ \text{i.o.}
	\right\}
\end{equation*}
and
$$
\nu_{\mathbf q}(E_k)\le \nu_{\mathbf q}\left\{
\omega:\frac{R_n(\omega)}{n}\ge \frac1k+\frac1n\ \text{i.o.}
\right\}=0.
$$
Together with
$$
\bigcup_{k \ge 1} E_k=
\left\{
\omega:\limsup_{n\to\infty}\frac{R_n(\omega)}{n} > 0
\right\},
$$
this yields
$$
\nu_{\mathbf q}\left\{
\omega:\limsup_{n\to\infty}\frac{R_n(\omega)}{n} > 0
\right\} = \nu_{\mathbf q}\left(\bigcup_{k \ge 1} E_k\right) \le \sum_{k\ge 1}\nu_{\mathbf q}(E_k)=0.
$$
Thus $\limsup\limits_{n\to\infty}\frac{R_n(\omega)}{n} = 0$ for $\nu_{\mathbf q}$-almost every $\omega\in\Sigma$, proving \eqref{eq:typical-symbolic-limits}.

The codings that are eventually identically $1$ or eventually identically $m$ form a countable set, which is a $\nu_{\mathbf q}$-measure zero set.
For all remaining codings satisfying \eqref{eq:typical-symbolic-limits}, Theorem~\ref{thm:local-scaling} applies directly and yields \eqref{eq:typical-local-exponent}.
\end{proof}

\begin{lemma}[\cite{Hutchinson1981}]
\label{lem:natural-measure}
Let $\mu':=\pi_*\nu_{\mathbf q}$ be the self-similar measure associated with the IFS $\{S_i\}_{i=1}^m$ and  the natural probability vector $\mathbf{q}=(r_1^s,\ldots,r_m^s)$. Then  $ \mu' =\frac{\restr{\mathcal H^s}{K}}{\mathcal H^s(K)}$. In particular, if $E\subset\Sigma$ satisfies $\nu_{\mathbf q}(E)=1$, then
$\mathcal H^s(K\setminus\pi(E))=0$.
\end{lemma}

\begin{lemma}[\cite{DMRV2006}]\label{lem:constant-exponent-distortion}
Let $(X,\rho)$ and $(Y,d)$ be metric spaces, and let $f:X\to Y$ be a homeomorphism.
Suppose that there exists a constant $\alpha\in(0,\infty)$ such that, for every non-isolated point $x$ of $X$,
\begin{equation*} 
  \lim_{\substack{y\to x}}
  \frac{\log d(f(x),f(y))}
       {\log\rho(x,y)}
  =\alpha.
\end{equation*}
Then, for every $A\subset X$,
\begin{equation*} 
  \dim_{\mathrm H}f(A)
  =\frac{1}{\alpha}\dim_{\mathrm H}A.
\end{equation*}
\end{lemma}

\begin{proof}[Proof of Theorem~\ref{thm:dimension-distortion}]
Let $\Omega_*$ be the set of codings that satisfy \eqref{eq:typical-symbolic-limits} and are not eventually identically
$1$ or $m$, and set $M:=\pi(\Omega_*)$.
By Proposition~\ref{prop:typical-exponent}, $\nu_{\mathbf q}(\Omega_*)=1$. The codings that are eventually identically $1$ or $m$ correspond precisely to the endpoints of complementary intervals of $K$, and therefore $M\subset K^*$, and $\mathcal H^s(K\setminus M)=0$ by Lemma~\ref{lem:natural-measure}.

We verify that the restriction
$\restr{F}{M}:M\to F(M)$ is a homeomorphism. Since the distribution
function $F$ is continuous and non-decreasing, its restriction to $M$ is
continuous. To prove injectivity, suppose that $x,y\in M$ satisfy $x<y$
and $F(x)=F(y)$. Then
\[
\mu((x,y])=F(y)-F(x)=0.
\]
Since $\operatorname{supp}\mu=K$, it follows that
$(x,y)\cap K=\varnothing$. Hence, $x$ and $y$ are the two endpoints of
the same complementary interval of $K$, contradicting the assumption that
$M\subset K^*$. Thus, $\restr{F}{M}$ is strictly increasing and
therefore is a continuous bijection from $M$ onto $F(M)$.

We now show that its inverse is continuous. Suppose that $x_n,x\in M$ and
$F(x_n)\to F(x)$, but $x_n\not\to x$. Then there exist $\varepsilon>0$
and a subsequence $\{x_{n_j}\}$ such that
\[
|x_{n_j}-x|\geq\varepsilon
\qquad\text{for all }j.
\]
Since $K\subset[0,1]$ is compact, by passing to a further subsequence, we
may assume that $x_{n_j}\to x'\in K$. Necessarily, $x'\neq x$. By the
continuity of $F$,
\[
F(x')=\lim_{j\to\infty}F(x_{n_j})=F(x).
\]
As above, $x$ and $x'$ must be the two endpoints of the same complementary
interval of $K$, contradicting $x\in M\subset K^*$. Therefore,
$(\restr{F}{M})^{-1}$ is continuous, and hence
$\restr{F}{M}:M\to F(M)$ is a homeomorphism.

By Proposition~\ref{prop:typical-exponent}, for every $x\in M$, 
$$
  \lim_{\substack{y\to x, y\in K}}
  \frac{\log|F(y)-F(x)|}{\log|y-x|}=\frac{h(\mathbf q,\mathbf p)}{\chi(\mathbf q)}.
$$
Restricting the approaching points to $M\subset K$ leaves the same limit unchanged. Hence Lemma~\ref{lem:constant-exponent-distortion} applies to
$f=\restr{F}{M}$ and proves the theorem.
\end{proof}

\section{A non-uniform three-branch example}\label{sec:example}

We illustrate the preceding results by considering a three-branch
self-similar system with non-uniform contraction ratios. Let
\[
S_1(x)=\frac{x}{9},
\qquad
S_2(x)=\frac{x}{9}+\frac{1}{3},
\qquad
S_3(x)=\frac{x}{3}+\frac{2}{3},
\]
and let $K$ denote the attractor of this IFS. Its similarity dimension $s={\log 2}/{\log 3}$. The associated natural probability vector is
\[
\mathbf q=(r_1^s,r_2^s,r_3^s)
=({1}/{4},{1}/{4},{1}/{2}).
\]

For each strictly positive probability vector
$\mathbf p=(p_1,p_2,p_3)$, define
\[
\alpha_{\mathbf p}
:=
\frac{h(\mathbf q,\mathbf p)}{\chi(\mathbf q)}
=
\frac{-\sum_{i=1}^3 q_i\log p_i}
{-\sum_{i=1}^3 q_i\log r_i}.
\]
Let $F_{\mathbf p}$ denote the generalized Cantor function associated
with the self-similar measure determined by $\mathbf p$. The quantity
$\alpha_{\mathbf p}$ describes the typical local scaling behavior of
$F_{\mathbf p}$ with respect to the natural geometry of $K$.

\begin{figure}[htbp]
	\centering
	\includegraphics[width=0.48\textwidth]
	{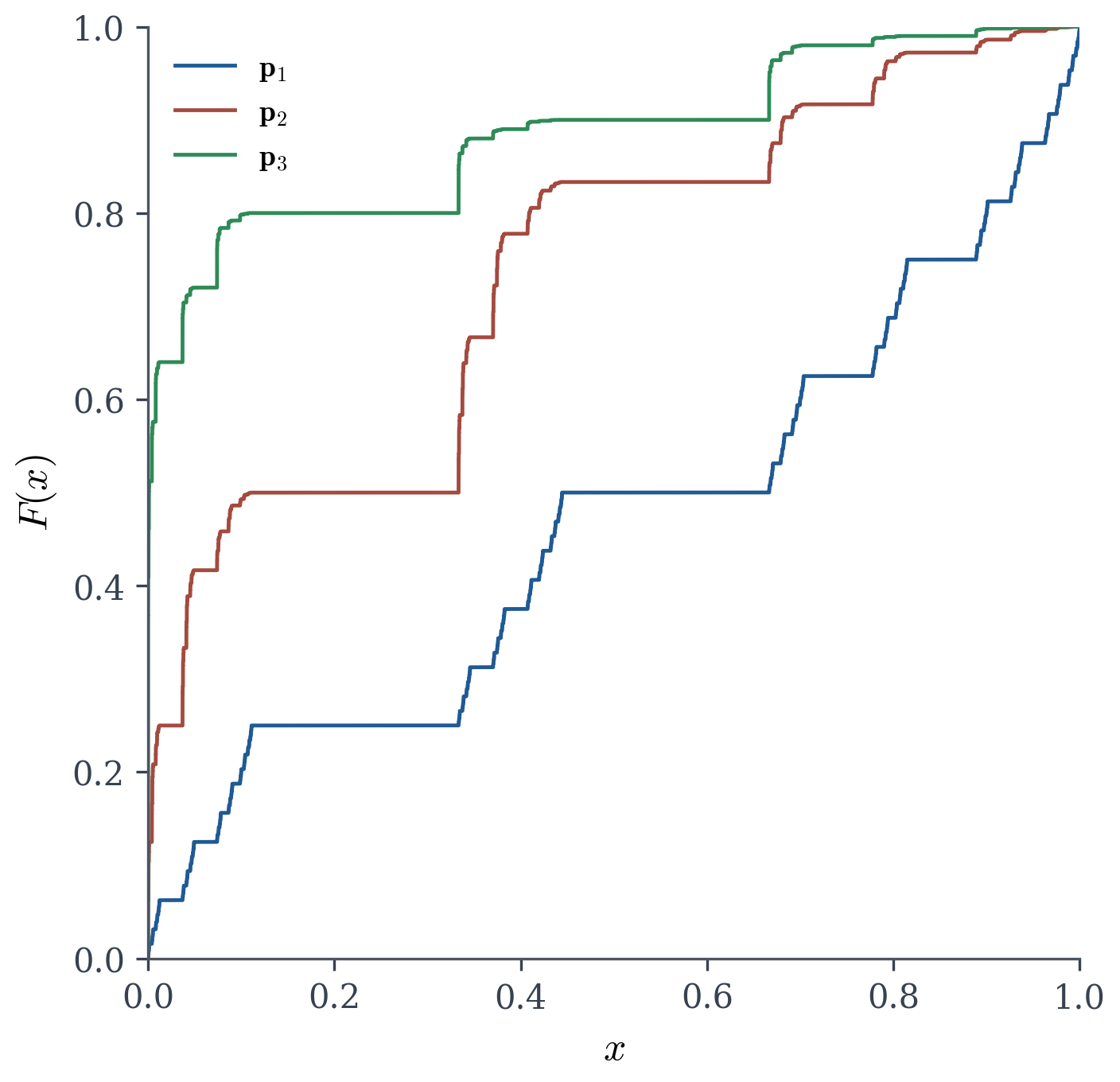}
	\caption{Generalized Cantor functions associated with the fixed IFS
		and three different probability vectors. Larger values of
		$\alpha_{\mathbf p}$ indicate flatter behavior at
		$\mathcal H^s$-almost every point of $K$.}
	\label{fig:cantor-alpha-three-curves}
\end{figure}

Consider the three probability vectors
\[
\mathbf p_1
=
\mathbf q
=(1/4,1/4,1/2),
\qquad
\mathbf p_2
=
(1/2,1/3,1/6),
\qquad
\mathbf p_3
=
(4/5,1/{10}, 1/{10}).
\] 
The corresponding generalized Cantor functions are shown in
Figure~\ref{fig:cantor-alpha-three-curves}. Since the underlying IFS is
fixed, their plateaus occur on the same complementary intervals of
$K$.  A direct calculation gives
$$\alpha_{\mathbf p_1}=0.6309\ldots, \quad	\alpha_{\mathbf p_2}=0.8155\ldots, \quad\text{and} \quad \alpha_{\mathbf p_3}=1.0818\ldots.$$

To interpret the exponents above, recall that a point
$x=\pi(\omega)\in K$ is called $\mathbf q$-typical if
\[
\lim_{n\to\infty}
\frac{\#\{1\leq k\leq n:\omega_k=i\}}{n}
=
q_i,
\qquad i=1,2,3.
\]
Thus, at a $\mathbf q$-typical point, the symbols $1$, $2$, and $3$ occur with asymptotic frequencies $1/4$, $1/4$, and $1/2$,
respectively. By the strong law of large numbers, the set of such points has full measure with respect to the natural self-similar
measure $\mu_{\mathbf q} (= \pi_* \nu_{\mathbf q})$. Since $q_i=r_i^s$, this set also has full $\mathcal H^s$-measure in $K$. More precisely, after excluding the null set on which the sublinear endpoint-run condition may fail, the local scaling exponent of
$F_{\mathbf p}$ at $\mathbf q$-typical points is $\alpha_{\mathbf p}$. Hence, on this full-measure set, a larger value of
$\alpha_{\mathbf p}$ corresponds to flatter local behavior of
$F_{\mathbf p}$.

For the three probability vectors above, the increase in
$\alpha_{\mathbf p}$ is accompanied by an increase in the mass assigned
to the leftmost first-level cylinder, from $1/4$ to $1/2$ and then to
$4/5$. Consequently, the initial rise of the corresponding distribution
function becomes progressively more pronounced. This visual comparison
does not, however, imply a pointwise ordering of the local flatness of
the three functions throughout $K$.

We conclude this section by showing that convergence of the bulk quotient
\(A_n(\omega)/B_n(\omega)\) alone does not guarantee the existence of the
local scaling exponent, thereby demonstrating the necessity of the sublinear
endpoint-run condition in  Theorem~\ref{thm:local-scaling}. 

Fix
\[
\mathbf p=\mathbf p_2=(1/2,1/3,1/6),
\qquad
\alpha_0=\frac{-\log p_1}{-\log r_1}
=\frac{\log 2}{2\log 3}.
\]
Consider the two codings $\omega^{(1)}$ and $\omega^{(2)}$ of the form:
\[
\omega_n^{(1)}=
\begin{cases}
2,&n=k^2\text{ for some }k\in\mathbb N,\\
1,&\text{otherwise},
\end{cases}
\qquad
\omega^{(2)}=2\,1^2\,2\,1^4\,2\,1^8\,2\,1^{16}\cdots .
\]
In both codings, the frequency of the symbol \(2\) tends to zero, and
hence that of the symbol \(1\) tends to one. Consequently,
\[
\frac{A_n(\omega^{(i)})}{n}\longrightarrow-\log p_1=\log 2,
\qquad i=1,2.
\]
Moreover, since both codings contain only the symbols \(1\) and \(2\)
and \(r_1=r_2=1/9\), we have
\[
B_n(\omega^{(i)})=2n\log3.
\]
It follows that
\begin{equation}\label{eq:example-common-bulk}
	\frac{A_n(\omega^{(i)})}{B_n(\omega^{(i)})}
	\longrightarrow
	\frac{\log2}{2\log3}
	=\alpha_0,
	\qquad i=1,2.
\end{equation}
Their endpoint runs, however, have different growth rates. For
$\omega^{(1)}$, consecutive occurrences of $2$ are separated by blocks
of $1$'s of length $2k$, and hence
$R_n(\omega^{(1)})=O(\sqrt n)$. Theorem~\ref{thm:local-scaling} therefore
implies that the local scaling exponent exists at
$x_1=\pi(\omega^{(1)})$ and equals $\alpha_0$.

For $\omega^{(2)}$, let $n_k=2^k+k-2$, the position of the $k$th
occurrence of $2$. This occurrence is followed by $2^k$ consecutive
$1$'s, so
\[
\frac{R_{n_k}(\omega^{(2)})}{n_k}
=\frac{2^k}{2^k+k-2}\longrightarrow1.
\]
We show directly that the local scaling exponent fails to exist at
$x_2=\pi(\omega^{(2)})$. Using the notation $Q$, $w$, and $z$ from the
proof of Theorem~\ref{thm:local-scaling}, if the first discrepancy occurs
at position $n$, then
\begin{equation}\label{eq:example-quotient-decomposition}
Q(x_2,y)=
\frac{B_{n-1}(\omega^{(2)})}
{B_{n-1}(\omega^{(2)})-\log w(x_2,y)}
\left(
\frac{A_{n-1}(\omega^{(2)})}{B_{n-1}(\omega^{(2)})}
+\frac{-\log z(x_2,y)}{B_{n-1}(\omega^{(2)})}
\right).
\end{equation}
The second term in parentheses records the correction produced by an
endpoint run.

Set $u_k=\omega^{(2)}|_{n_k-1}$ and define
\[
\eta^{(k)}=u_k\,1\,\underline{3},
\qquad y_k=\pi(\eta^{(k)}).
\]
The first discrepancy occurs at $n_k$, where $\omega^{(2)}$ has the
symbol $2$, followed by $2^k$ symbols $1$. The symbolic recursion for
$T$ gives
\[
z(x_2,y_k)=p_2T(\sigma^{n_k}\omega^{(2)})
\asymp p_2p_1^{2^k}\asymp2^{-2^k}.
\]
Moreover,
\[
B_{n_k-1}(\omega^{(2)})=2(n_k-1)\log3
\sim2^{k+1}\log3.
\]
Since $w(x_2,y_k)$ is bounded away from zero, the prefactor in
\eqref{eq:example-quotient-decomposition} tends to $1$. Together with
\eqref{eq:example-common-bulk}, this yields
\[
Q(x_2,y_k)\longrightarrow2\alpha_0.
\]

For a second sequence, choose the discrepancy near the middle of the
same block. Let
\[
m_k=n_k+2^{k-1},\qquad
\widetilde\eta^{(k)}
=\omega^{(2)}|_{m_k-1}\,2\,\underline{1},
\qquad
\widetilde y_k=\pi(\widetilde\eta^{(k)}).
\]
At least $2^{k-1}$ consecutive $1$'s remain in the coding of $x_2$ from
this position. Hence
\[
T(\sigma^{m_k-1}\omega^{(2)})\longrightarrow0,
\qquad T(2\,\underline{1})=p_1>0,
\]
and consequently
\[
z(x_2,\widetilde y_k)\longrightarrow p_1,
\qquad
\frac{-\log z(x_2,\widetilde y_k)}
{B_{m_k-1}(\omega^{(2)})}\longrightarrow0.
\]
Again the prefactor in \eqref{eq:example-quotient-decomposition} tends
to $1$, and therefore
\[
Q(x_2,\widetilde y_k)\longrightarrow\alpha_0.
\]

Both $y_k$ and $\widetilde y_k$ tend to $x_2$, but the corresponding
quotients have distinct limits. Thus the local scaling exponent at
$x_2$ does not exist, although the bulk quotient converges. This proves
that the endpoint-run condition $R_n(\omega)/n\to0$ cannot in general be
omitted from Theorem~\ref{thm:local-scaling}.

\end{document}